\documentclass[11pt]{amsart}
\usepackage[foot]{amsaddr}
\usepackage[paper]{zbs}

\DeclareMathOperator{\ch}{ch}
\newcommand{\calf}{\mathcal{F}}
\usepackage{tikz}
\usetikzlibrary{chains}
\tikzset{dot/.style={
circle,
fill,
inner sep=0pt,
minimum size=3pt
}}

\title{Coloring bipartite graphs with semi-small list size}
\author{Daniel G.\ Zhu}
\address{Massachusetts Institute of Technology, Cambridge, MA 02139}
\email{zhd@mit.edu}

\begin{document}
\begin{abstract}
Recently, Alon, Cambie, and Kang introduced asymmetric list coloring of bipartite graphs, where the size of each vertex's list depends on its part. For complete bipartite graphs, we fix the list sizes of one part and consider the resulting asymptotics, revealing an invariant quantity instrumental in determining choosability across most of the parameter space. By connecting this quantity to a simple question on independent sets of hypergraphs, we strengthen bounds when a part has list size 2. Finally, we state via our framework a conjecture on general bipartite graphs, unifying three conjectures of Alon-Cambie-Kang.
\end{abstract}

\maketitle

\section{Introduction}
Let $G = (V, E)$ be an undirected, finite, and simple graph. The concept of \emph{list coloring} was introduced independently by Erd\H{o}s-Rubin-Taylor \cite{Erdos1980} and Vizing \cite{Vizing1976}, who defined a graph to be \emph{$k$-choosable} if, for any set of colors $C$ and any \emph{list assignment} $L \colon V \to \binom{C}{k}$, there exists a coloring $c \colon V \to C$, with $c(v) \in L(v)$ for all $v \in V$, such that $c(v) \neq c(v')$ whenever $vv' \in E$. The \emph{choosability} $\ch(G)$ (also known as the \emph{list chromatic number}) is the minimum $k$ such that $G$ is $k$-choosable. Observe that $\ch(G)$ is at least the \emph{chromatic number} $\chi(G)$, by considering the case where $L$ is a constant function.

When $G$ is bipartite, $\chi(G) \leq 2$, but no such absolute bound holds for the choosability; in fact, Erd\H{o}s-Rubin-Taylor \cite{Erdos1980} showed in 1980 that $\ch(K_{n, n}) \sim \log_2 n$. After considering the choosability of random bipartite graphs, Alon and Krivelevich conjectured in 1998 that a similar bound holds for general bipartite graphs in terms of the maximum degree $\Delta$:
\begin{conj}[Alon-Krivelevich \cite{Alon1998}] \label{conj:ak}
If $G$ is bipartite, then $\ch(G) = O(\log \Delta)$.
\end{conj}
So far, relatively little progress has been made towards this conjecture beyond the trivial bound $\ch(G) = O(\Delta)$. In 1996, Johannson \cite{Johansson1996} (see also \cite{Molloy2002}) proved a bound of $O(\Delta/\log \Delta)$ for all triangle-free $G$, while in 2019 Molloy \cite{Molloy2019} found a different proof of this result, improving the constant.

Recently, Alon, Cambie, and Kang \cite{Alon2020} introduced an asymmetric variant of list coloring for bipartite graphs. Given a bipartite graph $G$ with bipartition $A \sqcup B$, we call $G$ \emph{$(k_A, k_B)$-choosable} if each list assignment $L \colon A \to \binom{C}{k_A}, B \to \binom{C}{k_B}$ admits a coloring, following the same rules as $k$-choosability. (In this paper every bipartite graph will come with an explicit ordered bipartition.) Letting $\Delta_A$ and $\Delta_B$ be the maximum degrees of vertices among $A$ and $B$, respectively, Alon, Cambie, and Kang derive various conditions for the choosability of complete bipartite graphs in terms of $k_A$, $k_B$, $\Delta_A$, $\Delta_B$. They then proceed to conjecture that similar bounds hold for general bipartite graphs, which can be construed as asymmetric generalizations of \cref{conj:ak}:
\begin{conj}[Alon-Cambie-Kang \cite{Alon2020}] \label{conj:ack}
Let $G$ be a bipartite graph with $\Delta_A, \Delta_B \geq 2$ and let $k_A$ and $k_B$ be positive integers. Then
\begin{parts}
\item[(a)] for any $\eps > 0$ there is a $\Delta_0$ such that $G$ is $(k_A, k_B)$-choosable whenever $k_A \geq \Delta_A^\eps$, $k_B \geq \Delta_B^\eps$, and $\Delta_A, \Delta_B \geq \Delta_0$;
\item[(b)] there exists an absolute $C > 0$ such that $G$ is $(k_A, k_B)$-choosable whenever $k_A \geq C \log \Delta_B$ and $k_B \geq C\log \Delta_A$;
\item[(c)] there exists an absolute $C > 0$ such that $G$ is $(k_A, k_B)$-choosable whenever $\Delta_A = \Delta_B = \Delta$ and $k_B \geq C(\Delta/\log \Delta)^{1/k_A} \log \Delta$.
\end{parts}
\end{conj}

\begin{rmk}
The condition that $\Delta_A, \Delta_B \geq 2$ is not present in \cite{Alon2020}; we make it here to avoid division-by-zero issues with $\log \Delta_A$ and $\log \Delta_B$.
\end{rmk}

While \cref{conj:ack} and its analogous theorem on complete bipartite graphs give asymptotic bounds on specific parts of the four-dimensional parameter space $(\Delta_A, \Delta_B, k_A, k_B)$, a more holistic treatment is lacking.

In this paper, we generalize Alon, Cambie, and Kang's work to apply to a much wider swath of the parameter space. We start by considering the choosability of complete bipartite graphs $G = K_{\Delta_B, \Delta_A}$ when $k_A$ is held fixed, finding that for each fixed $k_A$, the $(k_A, k_B)$-choosability is determined, up to a constant factor, by the quantity $\xi(\Delta_A, \Delta_B, k_A, k_B) = \Delta_B \log(\Delta_A)^{k_A-1}/k_B^{k_A}$. (We will henceforth omit the parameters of $\xi$ if clear from context.) Specifically, we show the following result:

\begin{thm} \label{thm:gensandwich}
For every positive integer $k_A$, there exist positive constants $C_1(k_A)$ and $C_2(k_A)$ such that, for all complete bipartite graphs $G = K_{\Delta_B, \Delta_A}$ and positive integers $k_B$ with $\Delta_A \geq k_A$ and $\Delta_B \geq k_B$, $G$ is $(k_A, k_B)$-choosable whenever $\xi < C_1$ and not $(k_A, k_B)$-choosable whenever $\xi > C_2$.
\end{thm}

Moreover, the quantity $\xi$ can often determine the $(k_A, k_B)$-choosability of a complete bipartite graph $G$ in an asymptotically tight manner. By defining $\xi_m(k_A)$ to be the infimum value of $\xi$ over all triples of positive integers $(k_B, \Delta_A, \Delta_B)$ such that $K_{\Delta_B,\Delta_A}$ is not $(k_A, k_B)$-choosable, we have the following results:
\begin{thm} \label{thm:lopsidedasymptotic}
Fix a positive integer $k_A$. For all $\Delta_A \geq k_A$ there exists a constant $c(\Delta_A)$ such that the smallest $\Delta_B$ such that $G = K_{\Delta_B, \Delta_A}$ is not $(k_A, k_B)$-choosable is asymptotic to $c(\Delta_A) k_B^{k_A}$ as $k_B \to \infty$. Moreover, $c(\Delta_A) \sim \xi_m(k_A)/\log(\Delta_A)^{k_A - 1}$ as $\Delta_A \to \infty$.
\end{thm}

\begin{thm} \label{thm:eqasymptotic}
Fix a positive integer $k_A$. The minimum $\Delta$ such that $K_{\Delta,\Delta}$ is not $(k_A, k_B)$-choosable is asymptotic to $\xi_m(k_A) k_B^{k_A}/(k_A\log(k_B))^{k_A - 1}$ as $k_B \to \infty$.
\end{thm}

The proofs of \cref{thm:lopsidedasymptotic} and \cref{thm:eqasymptotic} rely on procedures of \emph{graph amplification}, which take graphs that are not $(k_A, k_B)$-choosable to larger graphs that are not $(k_A, k_B')$-choosable, for some $k_B' > k_B$. These techniques are likely useful even when generalizing to the list coloring of bipartite graphs which are not complete.

Owing to their definition over a wide variety of possible graphs and lists, determining the value of $\xi_m(k_A)$ for even a single value of $k_A > 1$ is already an interesting question. Moreover, determining the value of $\xi_m(k_A)$ corresponds to a natural question involving the independent sets of hypergraphs. (Recall that at an independent set of a hypergraph is a subset of the vertices that does not contain any edge.)
\begin{lem} \label{lem:colorgraph}
The graph $G = K_{\Delta_B, \Delta_A}$ is $(k_A, k_B)$-choosable if and only if the following statement is true: ``For all $k_A$-uniform hypergraphs $H$ with $\Delta_B$ edges and families $\calf$ consisting of $\Delta_A$ subsets of size $k_B$ of the vertices of $H$, there is an independent set of $H$ that intersects every element in $\calf$.''
\end{lem}

In this paper we specifically address the case where $k_A = 2$, where the hypergraph reduces to an ordinary graph. Using probabilistic techniques, we obtain the following bounds:
\begin{thm} \label{thm:twobound}
$\frac{1}{2} \log 3 \leq \xi_m(2) \leq \log 2$.
\end{thm}
In particular, combining \cref{lem:colorgraph} and the lower bound of \cref{thm:twobound} yields the following result, which could be of independent interest:
\begin{cor}
Let $H$ be a graph with $e$ edges and let $\calf$ be a set of subsets of size $k$ of the vertices of $H$. If $\abs{\mathcal{F}} < 3^{\frac{k^2}{2e}}$, then there is an independent set of $H$ that intersects every element of $\calf$.
\end{cor}

After addressing bipartite list coloring when $k_A$ is fixed, we proceed to vary $k_A$ and examine the resulting asymptotic behavior of $\xi_m(k_A)$. Defining a quantity $\xi^*_m(k_A)$ which effectively acts as an upper bound on $\xi_m(k_A)$, we prove the following using similar methods of graph amplification:
\begin{thm} \label{thm:limxim}
The limit $\lim_{k_A \to \infty} \log (\xi^*_m(k_A)) / k_A$ exists.
\end{thm}

We conclude with a discussion of applicability to general bipartite graphs. Inspired by the above results, we make the following conjecture, which we show implies all three parts of \cref{conj:ack}: 
\begin{conj} \label{conj:xig}
There exist positive constants $\xi_g(k_A)$ such that all bipartite graphs $G$ satisfying $\xi < \xi_g(k_A)$ are $(k_A, k_B)$-choosable. Moreover, if the $\xi_g(k_A)$ are chosen to be the largest possible constants such that the above statement is true, then $\lim_{k_A \to \infty} \log (\xi_g(k_A)) / k_A$ exists.
\end{conj}

\subsection*{Outline}
\Cref{sec:general} contains basic results on the choosability of complete bipartite graphs. In \cref{sec:amplification} we introduce graph amplification techniques and develop properties of $\xi_m(k_A)$, proving \cref{thm:gensandwich}, \cref{thm:lopsidedasymptotic}, and \cref{thm:eqasymptotic}. \Cref{sec:two} focuses on bounding $\xi_m(2)$ and proves \cref{lem:colorgraph} and \cref{thm:twobound}, while in \cref{sec:metaasy} we study the asymptotic behavior of $\xi_m(k_A)$ and $\xi^*_m(k_A)$, proving \cref{thm:limxim}. We conclude with a discussion in \cref{sec:ext} regarding extensions to general bipartite graphs and discuss \cref{conj:xig}.

\section{Basic Results} \label{sec:general}
Fix positive integers $k_A$ and $k_B$. In this section, we let $G$ be the complete bipartite graph $K_{\Delta_B,\Delta_A}$, with our goal being to broadly determine the $(\Delta_A, \Delta_B)$ such that $G$ is $(k_A, k_B)$-choosable. We begin with some basic facts about choosability on complete bipartite graphs:
\begin{prop} \label{prop:trivchoosing}
If $\Delta_A < k_A$ or $\Delta_B < k_B$, then $G=K_{\Delta_B,\Delta_A}$ is $(k_A, k_B)$-choosable.
\end{prop}
\begin{proof}
Suppose $\Delta_A < k_A$. Given a list assignment $L$ on $G$, color every vertex $v$ in $B$ with an arbitrary element of $L(v)$. Then, for each vertex $v'$ in $A$, there are at most $\Delta_A$ colors that cannot be used to color $v'$, so since $\abs{L(v')} = k_A$ every vertex in $A$ can be colored.

If $\Delta_B< k_B$, a similar argument holds by swapping the roles of $A$ and $B$.
\end{proof}

Furthermore, observe that the set of $(\Delta_A, \Delta_B)$ such that $G$ is $(k_A,k_B)$-choosable is monotonic, in the sense that if $K_{\Delta_B,\Delta_A}$ is $(k_A,k_B)$-choosable, so is $K_{\Delta_B', \Delta_A'}$ for any $\Delta_B' \leq \Delta_B$ and $\Delta_A' \leq \Delta_A$.

The following proposition will be our main source of uncolorable complete bipartite graphs:
\begin{prop} \label{prop:counterexamples}
Let $r$, $k_A$, and $a_i$ be positive integers, for $1 \leq i \leq r$. Then, if $\Delta_A = k_A^r$ and $\Delta_B = \sum_i a_i^{k_A}$, then $G=K_{\Delta_B,\Delta_A}$ is not $(k_A, \sum_i a_i)$-choosable.
\end{prop}
\begin{proof}
We define a list assignment $L$ that admits no colorings. First define $k_A\sum_i a_i$ colors divided into $k_Ar$ ``blocks'' $C_i^{(j)}$ of size $a_i$, where $1 \leq i \leq r$ and $1 \leq j \leq k_A$. Now assign to the vertices in $A$ the $\sum_i a_i^{k_A}$ sets of $k_A$ colors consisting of a element from $C_i^{(j)}$ for all $j$ as $i$ is fixed. Assign to the vertices in $B$ the $k_A^r$ sets of $\sum_i a_i$ colors $C_1^{(e_1)} \cup C_2^{(e_2)} \cup \cdots \cup C_r^{(e_r)}$, for $(e_1, e_2, \ldots, e_r) \in \set{1, 2, \ldots, k_A}^r$.

Suppose $L$ admits a coloring $c$. Then, note that for all $i$ there cannot be vertices $v_1, v_2, \ldots, v_{k_A} \in B$ such that $c(v_j) \in C_i^{(j)}$ for all $j$, since that would contradict the vertex in $A$ with list equal to $\set{c(v_1), c(v_2), \ldots, c(v_{k_A})}$. Thus there exists a tuple $(e_1, e_2, \ldots, e_r) \in \set{1, 2, \ldots, k_A}^r$ such that, for all $i$, no vertex in $B$ is colored with an element of $C_i^{(e_i)}$. However, this contradicts the vertex with list $\bigcup_i C_i^{(e_i)}$.
\end{proof}

\begin{cor} \label{cor:simcount}
\cref{prop:counterexamples} admits the following special cases:
\begin{itemize}
\item The graph $G=K_{\Delta_B,\Delta_A}$ is not $(k_A, ar)$-choosable where $\Delta_A = k_A^r$ and $\Delta_B = a^{k_A} r$.
\item Moreover, $G$ is not $(k_A,k_B)$-choosable if either $(\Delta_A, \Delta_B) = (k_A, k_B^{k_A})$ or $(\Delta_A, \Delta_B) = (k_A^{k_B}, k_B)$.
\end{itemize}
\end{cor}
\begin{proof}
For the first part set all the $a_i$ equal to $a$. For the second part, further set $r = 1$ or $a = 1$, respectively.
\end{proof}

From this, it is easy to see that \cref{prop:trivchoosing} cannot be improved. Moreover, one can notice that, for fixed $k_A$, the ``interesting'' values of $\Delta_B$ grow polynomially in $k_B$, whereas $\Delta_A$ can grow exponentially.

By monotonicity, we are now able to show that all sufficiently large complete bipartite graphs are not $(k_A,k_B)$-choosable.

\begin{prop} \label{prop:smoothupper}
Suppose $\Delta_A \geq k_A$, $\Delta_B \geq k_B$, and $\Delta_B \log(\Delta_A)^{k_A - 1} > 2^{2k_A-1} \log(k_A)^{k_A - 1} k_B^{k_A}$. Then $G=K_{\Delta_B,\Delta_A}$ is not $(k_A, k_B)$-choosable.
\end{prop}
\begin{proof}
Let $r = \floor{\log_{k_A} \Delta_A}$ and $a = \ceil{k_B / r}$. If $r > k_B$ we are immediately done by the second part of \cref{cor:simcount} since $\Delta_A \geq k_A^{k_B}$ and $\Delta_B \geq k_B$. Thus we may assume otherwise, which implies that $r$ and $a$ are both positive integers.

By the first part of \cref{cor:simcount}, we now need to show that $k_A^r \leq \Delta_A$ and that $a^{k_A} r \leq \Delta_B$. The first is obvious, while the second follows from the two estimates $r > \log_{k_A}(\Delta_A)/2$ and $a < 2k_B/r$, which imply that
\[a^{k_A} r < \frac{2^{k_A} k_B^{k_A}}{r^{k_A - 1}} < \frac{2^{2k_A-1} k_B^{k_A} \log(k_A)^{k_A-1}}{\log(\Delta_A)^{k_A - 1}} < \Delta_B,\]
as desired.
\end{proof}

The constants involved here are probably far from tight. For example, when $k_A = 2$, more careful analysis yields the following bound, which improves the constant by a factor of about $4$.

\begin{prop} \label{prop:2smoothupper}
Suppose $\Delta_B \geq k$ and $\Delta_B \log \Delta_A > 1.4k^2$. Then $G=K_{\Delta_B,\Delta_A}$ is not $(2, k)$-choosable.
\end{prop}
\begin{proof}
Again let $r = \floor{\log_2 \Delta_A}$, and note that if $r > k$ we are done. Otherwise, we have $r \geq \frac{1}{\log_2 3}\log_2 \Delta_A = \frac{1}{\log 3}\log \Delta_A$.

Since $r \leq k$, we can let $a_i$ (for $1 \leq i \leq r$) be positive integers that sum to $k$ with the least possible variation. Letting $k = rq + r'$ with $0 \leq r' < r$, we obtain that
\[\sum_i a_i^2  = rq^2 + 2r'q + r' = \frac{k^2 + r'(r-r')}{r} \leq \frac{k^2}{r} + \frac{r}{4} \leq \frac{5}{4} \frac{k^2}{r}.\]

It is obvious that $2^r \leq \Delta_A$, so it suffices to show that $\sum_i a_i^2 \leq \Delta_B$. Indeed, we have
\[\Delta_B > \frac{1.4k^2}{\log \Delta_A} \geq \frac{1.4}{\log 3} \frac{k^2}{r} > \frac{5}{4} \frac{k^2}{r},\]
as desired.
\end{proof}

Having proved that $G$ is not $(k_A,k_B)$-choosable for large values of $(\Delta_A,\Delta_B)$, we conclude this section by showing that $G$ is $(k_A,k_B)$-choosable for small values of $(\Delta_A,\Delta_B)$.

\begin{lem} \label{lem:chernoff}
Fix some $k_A > 1$. Let $f(u) = 1 - u + u\log u$ and let the global maximum of the quantity $uf(u)^{k_A - 1}$ over the interval $u \in [0,1]$ be $\alpha(k_A)$, achieved at $u = u_0$. For every $\eps > 0$, there exists a $\Delta_0$, depending on $\eps$ and $k_A$, such that whenever $\Delta_A > \Delta_0$ and $\Delta_B \log(\Delta_A)^{k_A - 1} < (1-\eps) \alpha(k_A) k_B^{k_A}$, $G=K_{\Delta_B,\Delta_A}$ is $(k_A,k_B)$-choosable.
\end{lem}

The proof follows a probabilistic procedure found in \cite{Alon2020}, which we will repeat here for clarity. We will need the following Chernoff bound:
\begin{fact*}[see e.g.\ \cite{Hagerup1990}]
Let $X_1, \ldots, X_n$ be i.i.d.\ binary random variables that are $1$ with probability $p$. Then for $0 \leq \delta \leq 1$ we have
\[\setp\paren*{\sum_i X_i < (1-\delta) np} < e^{-npf(1-\delta)}.\]
\end{fact*}

\begin{proof}[Proof of \cref{lem:chernoff}]
Observe that the statement of the lemma gets strictly stronger when $\eps$ is decreased, so it suffices to prove the lemma for $\eps$ sufficiently small.

If $\Delta_B < k_B$ we are done by \cref{prop:trivchoosing}, so assume otherwise. Then we know that
\[\log \Delta_A < \paren*{\frac{(1-\eps)\alpha(k_A) k_B^{k_A}}{\Delta_B}}^{1/(k_A-1)} < \alpha(k_A)^{1/(k_A-1)} k_B.\] Set $p = \frac{1}{f(u_0) k_B}(1+\eps/k_A) \log \Delta_A$. Since
\[p < \frac{\alpha(k_A)^{1/(k_A-1)} k_B}{f(u_0) k_B} (1+\eps/k_A) = u_0^{1/(k_A - 1)}(1+\eps/k_A),\]
we find that $p < 1$ for sufficiently small $\eps$.

Given a list assignment $L$, we will color the vertices with the following procedure:
\begin{itemize}
\item Let $C'$ be a random subset of the set of colors $C$, chosen by independently placing every color $c \in C$ in $C'$ with probability $p$.
\item If there are at least $\frac{u_0}{f(u_0)} (1+\eps/k_A) \log \Delta_A$ vertices $v \in A$ with $L(v) \subseteq C'$, abort.
\item For each vertex $v$ in $A$ with $L(v) \subseteq C'$, choose an arbitrary color in $L(v)$ and remove it from $C'$. (Whether this is done simultaneously across all such vertices $v$ or sequentially is irrelevant.)
\item Color the vertices in $A$ with colors not in $C'$. Color the vertices in $B$ with colors in $C'$. If this is impossible, abort.
\end{itemize}
After the first step, the expected number of vertices in $A$ with no available colors is $p^{k_A} \Delta_B$. Thus the probability that this procedure aborts in the second step is, by Markov's inequality, at most
\begin{align*}
\frac{p^{k_A} \Delta_B}{\frac{u_0}{f(u_0)} (1+\eps/k_A) \log \Delta_A} &=  (1+\eps/k_A)^{k_A-1} \frac{ (\Delta_B \log (\Delta_A)^{k_A-1})/k_B^{k_A}}{u_0 f(u_0)^{k_A-1}} \\ &< (1+\eps/k_A)^{k_A - 1} (1-\eps).
\end{align*}
This quantity depends only on $\eps$ and is strictly less than $1$ for small $\eps$.

If the procedure does not abort in the second step, every vertex in $A$ has an available color in the last step. The probability that a fixed vertex $v$ in $B$ does not have an available color is bounded above by the probability that $\abs{L(v) \cap C'} < \frac{u_0}{f(u_0)} (1+\eps/k_A) \log \Delta_A = u_0k_Bp$ after the first step. This probability can be bounded with the aforementioned Chernoff bound, where we set $n = k_B$, $X_i$ to be $1$ if and only if a color is reserved for $B$ in the first step, and $1-\delta = u_0$. Combining this with a union bound over all $\Delta_A$ vertices in $B$, we find that the probability of some vertex in $B$ not having an available color after the procedure is at most
\[\Delta_A e^{-k_Bp f(u_0)} = \Delta_Ae^{-(1+\eps/k_A) \log \Delta_A} = \Delta_A^{-\eps/k_A}.\]
As $\Delta_A$ grows, this quantity can become arbitrarily small, so the total probability of failure, bounded above by $(1+\eps/k_A)^{k_A - 1} (1-\eps)+\Delta_A^{-\eps/k_A}$, must become less than $1$ for sufficiently large $\Delta_A$, as desired.
\end{proof}

To better contextualize \cref{prop:smoothupper} and \cref{lem:chernoff}, we introduce the following quantity:
\begin{defn}
Given $\Delta_A$, $\Delta_B$, $k_A$, $k_B$, define $\xi(\Delta_A, \Delta_B, k_A, k_B) = \Delta_B \log (\Delta_A)^{k_A - 1}/k_B^{k_A}$.
\end{defn}
In this notation, the condition in \cref{prop:smoothupper} simplifies to $\Delta_A \geq k_A$, $\Delta_B\geq k_B$, and $\xi > 2^{2k_A - 1} \log(k_A)^{k_A - 1}$, while the condition in \cref{lem:chernoff} becomes $\xi < (1-\eps)\alpha(k_A)$ for sufficiently large $\Delta_A$.

\section{Graph Amplification} \label{sec:amplification}
In this section we describe two ways to enlarge graphs which are distinct from the standard graph products. Through a variant of the ``tensor power trick'' we sharpen some of the bounds of the previous section and show the existence of an asymptotic across multiple regimes.

\subsection{Graph amplification procedures}
While in this paper we will only use graph amplification for complete bipartite graphs, we present definitions that are applicable to all bipartite graphs. Thus, for this subsection let $G$ be an arbitrary bipartite graph with bipartition $A \sqcup B$.
\begin{defn}
Let $r$ be a positive integer and let $G$ be a bipartite graph with bipartition $A \sqcup B$. The \emph{$r$-fold blowup} of $G$, denoted $G^{\curlyvee r}$, is the bipartite graph $G'$ with vertex parts $A' = A \times [r]$ and $B' = B^r$. Draw an edge between some vertex $(v, i) \in A'$ and $(v_1, \ldots, v_r) \in B'$ if and only if $vv_i$ is an edge in $G$.
\end{defn}
\begin{rmk}
This notion of blowup differs from notions of graph blowups defined elsewhere in the literature.
\end{rmk}
\begin{defn}
Let $r$ be a positive integer and let $G$ be a bipartite graph with bipartition $A \sqcup B$. The \emph{$r$-fold expansion} of $G$, denoted $G^{\curlywedge r}$, is the bipartite graph created by replacing each vertex in $A$ with $r$ copies of itself, connected to the same vertices in $B$.
\end{defn}
\begin{examp}
If $G$ is the path graph with $4$ vertices, its $2$-fold blowup and expansion are shown in \cref{fig:amp}.
\end{examp}
\begin{figure}[tbp]
\centering
\includegraphics{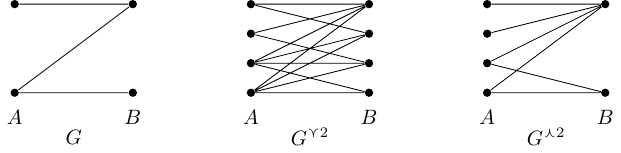}
\caption{$G$, $G^{\curlyvee 2}$, and $G^{\curlywedge 2}$ when $G$ is a path graph with $4$ vertices.} \label{fig:amp}
\end{figure}
These constructions are useful for the following reason:
\begin{lem} \label{lem:amplification}
If $G$ is not $(k_A, k_B)$-choosable, then neither $G^{\curlyvee r}$ nor $G^{\curlywedge r^{k_A}}$ are $(k_A, rk_B)$-choosable.
\end{lem}
\begin{proof}
Say $G$ has a list assignment $L$ on a set of colors $C$ that admits no colorings. We will define a list assignment $L'$ on $G^{\curlyvee r}$ using the colors $C \times [r]$, by setting $L'((v, i)) = \setmid{(c, i)}{c \in L(v)}$ for all $(v, i) \in A'$ and $L'((v_1, \ldots, v_r)) = \setmid{(c, i)}{c \in L(v_i), i \in [r]}$ for all $(v_1, \ldots, v_r) \in B'$.

Suppose $L'$ admits a coloring $c'$. For a fixed $i$, consider coloring every vertex $v \in  A$ with the color $c$ such that $c'((v, i)) = (c, i)$. Since this cannot be extended to a coloring of $L$, there must exist some $v_i \in B$ such that for all $c \in L(v_i)$ there exists an adjacent vertex $v'$ such that $(c, i) = c'((v', i))$. This creates a contradiction with the vertex $(v_1, \ldots, v_r) \in B'$. Thus $G^{\curlyvee r}$ is not $(k_A, rk_B)$-choosable.

Now, we deal with $G^{\curlywedge r^{k_A}}$. Call the parts $A' = A \times [r]^{k_A}$ and $B' = B$. We again define an unchoosable list assignment $L'$ with the colors $C \times [r]$.

For each $v \in A$ arbitrarily order the elements of $L(v)$ and call them $\ell_i(v)$, for $1 \leq i \leq k_A$. Then, for $j_1,j_2,\ldots,j_{k_A} \in [r]$, assign
\[L'((v, j_1, \ldots, j_{k_A})) = \set{(\ell_1(v), j_1), (\ell_2(v), j_2), \ldots, (\ell_{k_A}(v), j_{k_A})}.\]
For $v \in B'$, just set $L'(v) = L(v) \times [r]$.

Suppose $L'$ admits a coloring $c'$. Consider coloring each vertex $v \in B$ with the color $c$ such that $c'(v)=(c, j)$ for some $j$. This cannot be extended to a coloring of $L$, so there must exist some $v'\in A$ with neighbors $v_i \in B$, for all $1 \leq i \leq k_A$, satisfying $c'(v_i) = (\ell_i(v'), j_i)$ for some $j_i \in [r]$. Then there are no available colors for $(v', j_1, \ldots, j_{k_A}) \in A'$, meaning that $G^{\curlywedge r^{k_A}}$ is not $(k_A, rk_B)$-colorable.
\end{proof}

\subsection{Asymptotic existence for complete bipartite graphs}
For this subsection we return to the case when $G = K_{\Delta_B, \Delta_A}$ is complete bipartite. In this case \cref{lem:amplification} rewrites as follows:
\begin{cor} \label{cor:completeamp}
If $G = K_{\Delta_B, \Delta_A}$ is not $(k_A,k_B)$-choosable, then neither $G^{\curlyvee r} = K_{r\Delta_B, \Delta_A^r}$ nor $G^{\curlywedge r^{k_A}} = K_{r^{k_A} \Delta_B, \Delta_A}$ is $(k_A, rk_B)$-choosable.
\end{cor}
\begin{examp}
It is easy to see that $G = K_{1, k_A}$ is not $(k_A, 1)$-choosable. Therefore, $(G^{\curlyvee r})^{\curlywedge a^{k_A}} = K_{ra^{k_A},k_A^r}$ is not $(k_A, ra)$-choosable. This reproduces \cref{cor:simcount}.
\end{examp}

Based on this, we can now remove various conditions from \cref{lem:chernoff} to produce the following general statement:
\begin{lem} \label{lem:chernofffree}
If $\xi < \alpha(k_A)$, then $G=K_{\Delta_B,\Delta_A}$ is $(k_A,k_B)$-choosable.
\end{lem}
\begin{proof}
If $k_A = 1$, then it is easy to show that $G$ is $(k_A,k_B)$-choosable if and only if $\Delta_B < k_B$. Accordingly, $\xi = \Delta_B/k_B$ and $\alpha(1) = 1$. Now assume $k_A > 1$.

If $\Delta_A = 1$ we can apply \cref{prop:trivchoosing}. Otherwise, suppose $G = K_{\Delta_B, \Delta_A}$ is not $(k_A, k_B)$-choosable. Then, by \cref{cor:completeamp}, we have that $K_{r\Delta_B, \Delta_A^r}$ is not $(k_A, rk_B)$-choosable. However,
\[\xi(\Delta_A^r, r\Delta_B, k_A, rk_B) = \frac{r\Delta_B \log (\Delta_A^r)^{k_A}}{(rk_B)^{k_A-1}} = \xi(\Delta_A, \Delta_B, k_A, k_B) < \alpha(k_A),\]
which implies by \cref{lem:chernoff} that  $K_{r\Delta_B, \Delta_A^r}$ is $(k_A, rk_B)$-choosable for sufficiently large $r$. This is a contradiction.
\end{proof}

Combining \cref{prop:smoothupper} and \cref{lem:chernofffree} proves \cref{thm:gensandwich}.

Motivated by the ability to amplify a given unchoosable graph into larger ones, we define the quantity $\xi_m(k_A)$ to describe the ``smallest'' non-choosable graph. It will turn out to be crucial for understanding certain asymptotics.

\begin{defn}
If $k_A$ is a fixed positive integer, let $\xi_m(k_A)$ be the infimum of $\xi$ over all $(\Delta_A, \Delta_B, k_B)$ such that $K_{\Delta_B,\Delta_A}$ is not $(k_A, k_B)$-choosable.
\end{defn}

\begin{prop} \label{prop:ximbounds}
$\alpha(k_A) \leq \xi_m(k_A) \leq (\log k_A)^{k_A - 1}$.
\end{prop}
\begin{proof}
The lower bound follows from \cref{lem:chernofffree}. The upper bound follows from any part of \cref{cor:simcount}.
\end{proof}

We conclude this section with proofs of \cref{thm:lopsidedasymptotic} and \cref{thm:eqasymptotic}.
\begin{proof}[Proof of \cref{thm:lopsidedasymptotic}]
If $k_A = 1$ then $G$ is $(k_A,k_B)$-choosable if and only if $\Delta_B < k_B$. Thus $c(\Delta_A) = 1$ and $\xi_m(1) = 1$, so the theorem is true in this case. Now assume $k_A > 1$.

Fix a value of $\Delta_A \geq k_A$. Suppose that $G = K_{\Delta_B,\Delta_A}$ is not $(k_A, k_0)$-choosable if $\Delta_B = c_0k_0^{k_A}$, for a positive integer $k_0$ and positive real $c_0$. Then, by considering $G^{\curlywedge \ceil{k_B/k_0}^{k_A}}$, by \cref{cor:completeamp} we get that $K_{\Delta_B',\Delta_A}$ is not $(k_A, \ceil{k_B/k_0}k_0)$-choosable and thus not $(k_A,k_B)$-choosable if $\Delta'_B = c_0 \ceil{k_B/k_0}^{k_A} k_0^{k_A} \sim c_0k_B^{k_A}$.

Let $c(\Delta_A)$ be the infimum of $c_0$ over all such $c_0, k_0$. By the definition of $c(\Delta_A)$, the minimum $\Delta_B$ such that $G$ is not $(k_A,k_B)$-choosable is at least $c(\Delta_A)k_B^{k_A}$. Moreover, by the previous paragraph, it is, for all $\eps > 0$, eventually bounded above by $(c(\Delta_A) + \eps) k_B^{k_A}$. Since $\xi_m(k_A) > 0$, $c(\Delta_A) > 0$, so the minimum $\Delta_B$ is indeed asymptotic to $c(\Delta_A) k_B^{k_A}$.

Now we consider varying $\Delta_A$. By monotonicity we have that $c(\Delta_A)$ is nonincreasing in $\Delta_A$. Moreover, suppose that $G = K_{\Delta_B,\Delta_A}$ is not $(k_A, k_B)$-choosable for a given $\Delta_A$ and $\Delta_B = c_0 k_B^{k_A}$. Then, for any positive integer $r$, by \cref{cor:completeamp} we get that $G^{\curlyvee r} = K_{c_0 r k_B^{k_A}, \Delta_A^r}$ is not $(k_A, rk_B)$-choosable. We conclude that $c(\Delta_A^r) \leq c(\Delta_A)/r^{k_A - 1}$.

The definition of $\xi_m(k_A)$ may be rearranged to
\[\xi_m(k_A) = \inf_{\Delta_A \geq k_A} c(\Delta_A) \log (\Delta_A)^{k_A - 1},\] so  $c(\Delta_A) \geq \xi_m(k_A) /\log (\Delta_A)^{k_A - 1}$. Moreover, given some integer $\Delta_0 \geq k_A$,
\[c(\Delta_A) \leq c(\Delta_0^{\floor*{\frac{\log \Delta_A}{\log \Delta_0}}}) \leq c(\Delta_0)/\floor*{\frac{\log \Delta_A}{\log \Delta_0}}^{k_A - 1} \sim \frac{c(\Delta_0) \log (\Delta_0)^{k_A - 1}}{\log (\Delta_A)^{k_A - 1}},\]
so for every $\eps > 0$ we have $c(\Delta_A) \leq (\xi_m(k_A) + \eps)/\log (\Delta_A)^{k_A-1}$ for sufficiently large $\Delta_A$. Since $\xi_m(k_A) > 0$, this proves $c(\Delta_A) \sim \xi_m(k_A)/\log (\Delta_A)^{k_A-1}$.
\end{proof}

\begin{proof}[Proof of \cref{thm:eqasymptotic}]
For a positive integer $\Delta$, let $k(\Delta)$ be the maximum positive integer $k_B$ such that $K_{\Delta,\Delta}$ is not $(k_A, k_B)$-choosable, or $0$ if no such $k_B$ exists. Note that $k(\Delta)$ must exist since $K_{\Delta,\Delta}$ is $(k_A, \Delta+1)$-choosable. We claim that $k(\Delta) \sim (\Delta \log(\Delta)^{k_A - 1}/\xi_m(k_A))^{1/k_A}$.

We first deduce the theorem assuming this claim. For a positive integer $i$, let $\Delta_i$ be the minimum $\Delta$ such that $K_{\Delta,\Delta}$ is not $(k_A, i)$-choosable. By monotonicity, $\Delta_i$ is also the minimum $\Delta$ such that $i \leq k(\Delta)$. Hence, if $\Delta_i > 1$, $k(\Delta_i - 1) < i \leq k(\Delta_i)$. Moreover, since $k(\Delta)$ exists for all $\Delta$, we also have $\lim_{i \to \infty} \Delta_i = \infty$.

Since
\[1 = \lim_{i\to\infty} \frac{k(\Delta_i - 1)}{(\Delta_i \log(\Delta_i)^{k_A - 1}/\xi_m(k_A))^{1/k_A}} = \lim_{i\to\infty} \frac{k(\Delta_i)}{(\Delta_i \log(\Delta_i)^{k_A - 1}/\xi_m(k_A))^{1/k_A}},\]
we conclude that
\[1 = \lim_{i\to\infty} \frac{i}{(\Delta_i \log(\Delta_i)^{k_A - 1}/\xi_m(k_A))^{1/k_A}}. \tag{$*$} \label{eq:Deltaiasymp}\]
The statement of \cref{thm:eqasymptotic} is equivalent to
\[1 = \lim_{i\to\infty} \frac{\xi_m(k_A) i^{k_A}}{(k_A \log(i))^{k_A - 1} \Delta_i}.\]
To finish, it suffices to show that $\log(\Delta_i) \sim k_A \log(i)$, which is apparent by taking the logarithm of (\ref{eq:Deltaiasymp}).

We now prove that our claim that $k(\Delta) \sim (\Delta \log(\Delta)^{k_A - 1}/\xi_m(k_A))^{1/k_A}$. First of all, by the definition of $\xi_m(k_A)$,
\[\frac{\Delta \log(\Delta)^{k_A - 1}}{k(\Delta)^{k_A}} \geq \xi_m(k_A) \iff k(\Delta) \leq (\Delta \log(\Delta)^{k_A - 1}/\xi_m(k_A))^{1/k_A}.\]

Now, suppose that $G = K_{\Delta_B, \Delta_A}$ is not $(k_A, k_B)$-choosable. Letting
\[r_1 = \floor*{\paren*{\frac{\Delta \log(\Delta_A)}{\log(\Delta) \Delta_B}}^{1/k_A}} \text{ and } r_2 = \floor*{\frac{\log(\Delta)}{\log(\Delta_A)}},\]
we obtain that $(G^{\curlywedge r_1^{k_A}})^{\curlyvee r_2} = K_{r_1^{k_A}r_2 \Delta_B, \Delta_A^{r_2}}$ is not $(k_A, r_1r_2k_B)$-choosable. However, it is easily checked that $r_1^{k_A}r_2\Delta_B, \Delta_A^{r_2} \leq \Delta$, so $k(\Delta) \geq r_1r_2 k_B$. In particular,
\[\liminf_{\Delta\to\infty} \frac{k(\Delta)}{(\Delta \log(\Delta)^{k_A - 1}/\xi_m(k_A))^{1/k_A}} \geq \frac{k_B \xi_m(k_A)^{1/k_A}}{\Delta_B^{1/k_A} \log(\Delta_A)^{1-1/k_A}},\]
which can be made arbitrarily close to $1$ by the definition of $\xi_m(k_A)$. This concludes the proof.
\end{proof}

\begin{rmk}
The statements of \cref{thm:lopsidedasymptotic} and \cref{thm:eqasymptotic} imply that the threshold between choosable and non-choosable graphs occurs at $\xi \approx \xi_m(k_A)$ in the regimes $\Delta_B \gg \Delta_A$ and $\Delta_A \approx \Delta_B$. Similar techniques will imply the same whenever $\Delta_A$ grows subexponentially in $\Delta_B$. As an example of when $\xi_m(k_A)$ ceases to be relevant, observe that for fixed $\Delta_B$, the largest $k_B$ such that $K_{\Delta_B, \Delta_A}$ is not $(k_A,k_B)$-choosable never exceeds $\Delta_B$.
\end{rmk}

\section{Independent Sets and Bounds on \texorpdfstring{$\xi_m(2)$}{ξ\_m(2)}} \label{sec:two}
In this section we phrase the $(k_A, k_B)$-choosability of complete bipartite graphs in terms of the set avoidance of independent sets of a ``color hypergraph'' $H$. We then use this formulation to prove \cref{thm:twobound}.

\subsection{The color hypergraph}
In this section we prove \cref{lem:colorgraph}. The main idea is that, given a $(k_A, k_B)$-list on a complete bipartite graph, the $k_A$-sets of colors assigned to vertices in $A$ have a natural ``graph-like'' structure.

\begin{proof}[Proof of \cref{lem:colorgraph}]
Call a list assignment $L$ on $G$ \emph{maximal} if no two vertices in $A$ or no two vertices in $B$ are assigned the same set of colors. It is easy to see that if $G$ has an unchoosable list assignment, than it has an unchoosable maximal list assignment as well, so as far as the choosability of $G$ is concerned, non-maximal lists can be ignored.

Given a maximal list assignment $L$ on $G$, define $H(L)$ to be the hypergraph with set of vertices $C$ and edges $\setmid{L(v)}{v\in A}$. Also, let $\calf(L)$ be the set of subsets of size $k_B$ of $C$ that are $L(v)$ for some $v \in B$. 

Since any $k_A$-uniform hypergraph $H$ with $\Delta_B$ edges and family $\calf$ of $\Delta_A$ sets of size $k_B$ of the vertices of $H$ can be represented as $H(L)$ and $\calf(L)$ for some list assignment $L$, it suffices to show that $L$ is unchoosable if and only if every independent set in $H(L)$ is disjoint from an element of $\calf(L)$. If there exists a coloring $c$ consistent with $L$, then $I = \setmid{c(v)}{v \in B}$ intersects every element of $\calf(L)$, but is an independent set of $H$ since for all $v \in A$ there must exist some color in $L(v)$ not in $I$.

Conversely, if $I$ is an independent set of $H$ intersecting every subset in $\calf$, then for every vertex $v \in A$ the set $L(v)$ intersects the complement of $I$, while for every vertex $v \in B$ the set $L(v)$ intersects $I$. Thus one may construct a valid coloring of $G$ by coloring each vertex $v \in A$ with element of $L(v) \setminus I$ and each vertex in $v \in B$ with an element of $L(v) \cap I$.
\end{proof}

\begin{examp}
If $H = K_{k, k}$, then every independent set in $H$ must be disjoint from one of its parts (or both if it is empty). Therefore $K_{k^2, 2}$ is not $(2, k)$-choosable. More generally, in the case where $k_A = 2$, the construction in \cref{cor:simcount} corresponds to $H$ being $r$ disjoint copies of $K_{a, a}$, which can be shown to not be $(2, ar)$-choosable by letting $\calf$ be the $2^r$ ways to choose one part from each copy of $K_{a,a}$.
\end{examp}

\subsection{A probabilistic algorithm}
For the remainder of this section we prove the lower bound in \cref{thm:twobound}, as the upper bound has already been proven in \cref{prop:ximbounds}. Specifically, given a graph $H$ and a subset family $\calf$, we will probabilistically construct an independent set  that intersects every element of $\calf$. The algorithm used to construct this independent set $I$ is actually quite simple: one chooses a uniformly random order of the vertices and then processes every vertex in that order. The action of \emph{processing} a vertex $v$ consists of adding $v$ to $I$ if none of its neighbors are already in $I$.

In order to prove \cref{thm:twobound} we need to analyze the probability that $I$ is disjoint from any given set $S$.

\begin{prop} \label{prop:probsetup}
If $I \cap S = \emptyset$, then every vertex in $S$ has a neighbor not in $S$ that was processed before it.
\end{prop}
\begin{proof}
Take a vertex $v \in S$. Since $v \notin I$, at the time of processing $v$, some vertex $v'$ in its neighborhood must have been previously processed and placed in $I$. It cannot be the case that $v' \in S$, so we must have $v' \in N(v) \setminus S$.
\end{proof}

For any given subset $S'$ of the vertices of $H$, the order in which the vertices in $S'$ are processed follows a uniform distribution on the set of orderings of $S'$. Hence, \cref{prop:probsetup} allows the probability that $I$ and $S$ are disjoint to be upper-bounded by a probability that is dependent only on the structure of $H$ ``near $S$''. Specifically, we make the following definitions:
\begin{defn}
Let $H'$ be a bipartite graph with bipartition $A \sqcup B$. Given a order on the vertices of $H'$, call a vertex $v \in A$ \emph{shadowed} if it has a neighbor in $H'$ that precedes it. Let $p(H')$ be the probability that every vertex in $A$ is shadowed under a uniformly random order of the vertices of $H'$.
\end{defn}
\begin{defn}
For a subset $S$ of the vertices of $H$, let $T_S$ denote the set $\bigcup_{v \in S} N(v) \setminus S$. Let $H_S$ be the bipartite subgraph of $H$ with bipartition $S \sqcup T_S$, consisting of all edges in $H$ between $S$ and $T_S$.
\end{defn}
In this language, we now have the following corollary of \cref{prop:probsetup}:
\begin{cor} \label{cor:phs}
The probability that $I \cap S = \emptyset$ is at most $p(H_S)$.
\end{cor}

We now proceed to bound $p(H_S)$, with our main tool being the following recursion:
\begin{lem} \label{lem:precursion}
Let $H'$ be a nonempty bipartite graph with bipartition $A \sqcup B$. Then $p(H') = \frac{1}{\abs{A} + \abs{B}} \sum_{v \in B} p(H' \setminus N[v])$, where $H' \setminus N[v]$ denotes the subgraph of $H'$ obtained by deleting the vertices in the closed neighborhood $N[v]$.
\end{lem}
\begin{proof}
Consider the first vertex $v$ in $H'$ in the order. If it is in $A$, it cannot be shadowed. Otherwise, $v \in B$ shadows all of its neighbors, so whether all other vertices in $A$ are shadowed only depends on the relative ordering of the vertices in $H'$ that are not $v$ or a neighbor of $v$. Since any such ordering is equally likely, the probability that every vertex in $A$ is shadowed conditional on $v$ being chosen first is $p(H' \setminus N[v])$. This concludes the proof.
\end{proof}
\begin{lem} \label{lem:fancybound}
Let $H', A, B$ be as in \cref{lem:precursion}, with the additional condition that $A$ is nonempty. If $H'$ has no edges then $p(H') = 0$. Otherwise
\[p(H') \leq \paren*{1 + \frac{\abs{A} \Delta}{\abs{E}}}^{-\abs{A}/\Delta},\]
where $\abs{E}$ is the number of edges of $H'$ and $\Delta$ is the maximum degree among any vertex in $B$.
\end{lem}
\begin{rmk} \label{rmk:fancyboundadd}
When applying \cref{lem:fancybound}, we will frequently use the easy-to-show fact that this bound is increasing in $\Delta$, meaning the bound of \cref{lem:fancybound} holds even when $\Delta$ is replaced with any positive integer $r \geq \Delta$.
\end{rmk}

Before we state the proof of \cref{lem:fancybound}, we need a quick result about quantities related to the degrees of vertices in $B$:
\begin{defn}
For a given $H'$ with bipartition $A \sqcup B$ and $\abs{B} = \ell$, let $v_1, v_2, \ldots, v_\ell$ be the vertices of $B$. Let $d_i$ be the degree of $v_i$ and $D_i$ be the sum of the degrees of the neighbors of $v_i$.
\end{defn}

\begin{prop} \label{prop:extdegree}
$\sum_i d_i^2/D_i \leq \abs{A}$.
\end{prop}
\begin{proof}
Since the function $x \mapsto 1/x$ is convex, by Jensen's inequality we have
\[\sum_{u \in N(v_i)} \frac{1}{\deg u} \geq \frac{d_i}{D_i/d_i}.\]
Now sum over all $i$. The left hand side becomes
\[\sum_i \sum_{u \in N(v_i)} \frac{1}{\deg u} = \sum_{u \in A'} \frac{\deg u}{\deg u} \leq \abs{A},\]
where $A'$ is the set of non-isolated vertices in $A$.
\end{proof}

We will also need the following inequality. As the proof is rather tedious, we defer it to \cref{append:ineq}.

\begin{lem} \label{lem:tedious}
If $a$, $b$, $\beta$, and $\gamma$ are nonnegative reals satisfying $a \leq 1$ and $\gamma > \max(a,b)$, then
\[
\paren*{1+\beta\frac{\gamma - a}{\gamma - b}}^{-(\gamma-a)} \leq (1+\beta)^{-\gamma}(1 + \beta a^2/b).
\]
\end{lem}

\begin{proof}[Proof of \cref{lem:fancybound}]
The case where $H'$ has no edges is obvious, so assume otherwise.

Consider a counterexample $H'$ that minimizes $\abs{B}$. If there is an isolated vertex in $A$, then $p(H') = 0$ and the result is trivial, so assume otherwise.

By \cref{lem:precursion}, we have
\[
p(H') = \frac{1}{\abs{A} + \abs{B}} \sum_i p(H' \setminus N[v_i]).
\]
We now claim that
\[p(H' \setminus N[v_i]) \leq \paren*{1+\frac{\abs{A}\Delta}{\abs{E}}}^{-\abs{A}/\Delta} \paren*{1 + \frac{d_i^2}{D_i}},\]
which implies the result by \cref{prop:extdegree}. In the case where $N(v_i) = A$, we have $d_i = \abs{A}$, $D_i = \abs{E}$, and $\Delta = \abs{A}$, so we wish to prove that
\[1 \leq \paren*{1 + \frac{\abs{A}^2}{\abs{E}}}^{-1} \paren*{1 + \frac{\abs{A}^2}{\abs{E}}},\]
which is true. Otherwise, since $A$ has no isolated vertices, $H' \setminus N[v_i]$ has at least one edge. By the minimality of $H'$ (and \cref{rmk:fancyboundadd}) we therefore have
\[p(H' \setminus N[v_i]) \leq \paren*{1 + \frac{(\abs{A}-d_i)\Delta}{\abs{E}-D_i}}^{-(\abs{A}-d_i)/\Delta},\]
so it remains to show that
\[\paren*{1 + \frac{(\abs{A}-d_i)\Delta}{\abs{E}-D_i}}^{-(\abs{A}-d_i)/\Delta} \leq \paren*{1+\frac{\abs{A}\Delta}{\abs{E}}}^{-\abs{A}/\Delta} \paren*{1 + \frac{d_i^2}{D_i}}.\]
This inequality follows from \cref{lem:tedious} upon substituting $\beta = \abs{A}\Delta/\abs{E}$, $\gamma = \abs{A}/\Delta$, $a = d_i/\Delta$, and $b = D_i \abs{A}/(\abs{E}\Delta)$.
\end{proof}

\begin{proof}[Proof of \cref{thm:twobound}]
The upper bound is due to \cref{prop:ximbounds}. To show the lower bound, we need to show that $K_{\Delta_B,\Delta_A}$ is $(2,k)$-choosable as long as $\Delta_B \log(\Delta_A)/k^2 < \frac{1}{2} \log 3$, which is equivalent to $\Delta_A < 3^{k^2/(2\Delta_B)}$. (Here $k$ serves the role of $k_B$.) By \cref{lem:colorgraph}, we wish to show that for every graph $H$ with $\Delta_B$ edges and a family $\calf$ of $\Delta_A< 3^{k^2/(2\Delta_B)}$ vertex subsets of size $k$, there exists an independent set in $H$ intersecting every element of $\calf$.

Consider the process of iteratively deleting the vertex in $H$ with maximum degree. Suppose the maximum degree after $i$ deletions is $\Delta_i$. Note that after $j$ deletions, the number of edges remaining is $\Delta_B - \sum_{0\leq i < j} \Delta_i$.

We claim that there must exist some $0 \leq i < k$ with $\Delta_i \leq (2k-2i-1)\Delta_B/k^2$ and $\Delta_{i'} > (2k-2i'-1)\Delta_B/k^2$ for all $0 \leq i' < i$. If such an $i$ did not exist, we must have $\Delta_i > (2k-2i-1)\Delta_B/k^2$ for all $0 \leq i < k$, which implies that after $k$ deletions the number of edges is less than $\Delta_B - \sum_{0 \leq i < k}(2k-2i-1)\Delta_B/k^2 = 0$, a contradiction. Note that $H$ has at least $k$ vertices since the elements of $\calf$ have size $k$.

If we assume that $i$ has the above properties, then after $i$ deletions, which yields a graph $H'$, the number of edges is at most
\[\Delta_B - \sum_{0 \leq j<i} \frac{(2k-2j-1) \Delta_B}{k^2} = \frac{\Delta_B(k-i)^2}{k^2}\] and the maximum degree is at most $2(k-i)\Delta_B/k^2$. We choose a random independent set $I$ of $H'$ according to our procedure. Since each set in $\calf$ shares at least $k-i$ vertices with $H'$, by \cref{cor:phs} and \cref{lem:fancybound} the probability that $I$ is disjoint from any fixed set of $\calf$ is at most
\[\paren*{1 + \frac{(k-i)2(k-i)\Delta_B/k^2}{\Delta_B(k-i)^2/k^2}}^{-(k-i)/(2(k-i)\Delta_B/k^2)} = 3^{-k^2/(2\Delta_B)}.\]
Here we have used the fact that the bound $(1+\frac{\abs{A}\Delta}{\abs{E}})^{-\abs{A}/\Delta}$ of \cref{lem:fancybound} is increasing in $\Delta$ (as remarked in \cref{rmk:fancyboundadd}) and $\abs{E}$. Thus some choice of $I$ intersects every element of $\calf$, as desired.
\end{proof}

Examining the proof of \cref{lem:fancybound} suggests that, up to isolated vertices, equality holds only if $H'$ is the disjoint union of several identical complete bipartite graphs. However, such a situation is seemingly contradictory with the fiducial estimate of $\Delta_i \approx 2(k-i)\Delta_B/k^2$. Therefore, it appears unlikely that $\xi_m(2) = \frac{1}{2} \log 3$ in reality. In fact, we make the following conjecture:
\begin{conj} \label{conj:xim2}
$\xi_m(2) = \log 2$.
\end{conj}
One possible way to approach \cref{conj:xim2} is to strengthen \cref{lem:fancybound} so that it relies solely on local parameters of the graph $H'$, instead of the global parameter $\Delta$. In particular, the following bound seems to be true:
\begin{conj} \label{conj:localhs}
Let $H', A, B$ be as in \cref{lem:precursion}, such that there are no isolated vertices in $B$. If $M$ is the minimum of $\sum_i e_i \log(1+e_i)/d_i$ over all nonnegative reals $e_i$ ($1 \leq i \leq \abs{B}$) that sum to $\abs{A}$, then $p(H') \leq e^{-M}$.
\end{conj}

A useful set of ``local'' parameters that sum to $\abs{A}$ are defined by $f_i = \sum_{u \in N(v_i)} 1/\deg u$, so one might think to prove the following stronger statement:
\[p(H') \leq \prod_i (1+f_i)^{-f_i/d_i}.\]
However, it is false, with the following graph as a counterexample:
\begin{center}
\includegraphics{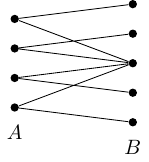}
\end{center}

\section{The Behavior of \texorpdfstring{$\xi_m(k_A)$}{ξ\_m(k\_A)}} \label{sec:metaasy}
\Cref{conj:xim2} may seem to suggest that $\xi_m(k_A) = \log(k_A)^{k_A - 1}$. However, in this section, we define $\xi^*_m(k_A)$ and we prove \cref{thm:limxim}, which disproves this conjecture for large $k_A$. We also show that $\xi_m(k_A) < \log(k_A)^{k_A - 1}$ for smaller values of $k_A$.

\subsection{Concrete bounds on \texorpdfstring{$\xi_m(k_A)$}{ξ\_m(k\_A)}}
Before considering asymptotics, we prove that $\xi_m(k_A) < \log(k_A)^{k_A - 1}$ for many values of $k_A$.
\begin{prop}
$\xi_m(3) < (\log 3)^2$.
\end{prop}
\begin{proof}
Erd\H{o}s-Rubin-Taylor \cite{Erdos1980} showed that $K_{7, 7}$ is not $(3, 3)$-choosable. So
\[\xi_m(3) \leq \frac{7 \log(7)^2}{3^3} < \log(3)^2.\qedhere\]
\end{proof}

We even have the following result:
\begin{prop} \label{prop:compositebad}
If $k_A$ is composite, then $\xi_m(k_A) < (\log k_A)^{k_A-1}$.
\end{prop}

The proof of \cref{prop:compositebad} requires a quick fact about exponential functions:
\begin{lem} \label{lem:exppts}
Let $a$ and $b$ be positive and let $g(x) = be^{-ax}$. Then there are at most three reals $x$ with $g(g(x)) = x$.
\end{lem}
This is proven in \cref{append:ineq}.

\begin{proof}[Proof of \cref{prop:compositebad}]
Write $k_A = ar$ with $a \geq r > 1$. By \cref{cor:simcount}, $K_{\Delta_B, \Delta_A}$ is not $(k_A,k_A)$-choosable when $\Delta_B = a^{k_A} r$ and $\Delta_A = k_A^r$. Observe that
\[\Delta_B \geq \sqrt{k_A}^{k_A} \sqrt{k_A} = k_A^{(k_A+1)/2} > k_A^{\sqrt{k_A}} \geq \Delta_A.\]
By symmetry $K_{\Delta_A,\Delta_B}$ is not $(k_A,k_A)$-choosable; we claim that
\[\frac{\Delta_A \log(\Delta_B)^{k_A-1}}{k_A^{k_A}} < \log(k_A)^{k_A-1}.\]
From here it will be useful to work with $\delta_A = \log \Delta_A$ and $\delta_B = \log \Delta_B$. Letting $\alpha = \frac{1}{k_A-1}$ and $\beta = k_A^{k_A/(k_A-1)} \log(k_A)$, we have $\delta_A = \beta e^{-\alpha \delta_B}$, while we wish to prove that $\delta_B < \beta e^{-\alpha \delta_A}$. Thus, letting $g(x) = \beta e^{-\alpha x}$, we want to show that $\delta_B < g(g(\delta_B))$.

At this point, we observe that that $g$ swaps $\log k_A$ and $k_A \log k_A$. Since $g$ is monotonically decreasing, it has a fixed point $x_0$ in between these two numbers. We claim that $x_0 < \delta_B < k_A \log k_A$. To show the first inequality, we use the fact $g(\delta_B) - \delta_B = \delta_A - \delta_B < 0 = g(x_0) - x_0$, which suffices as $x \mapsto g(x) - x$ is a decreasing function. The second inequality is equivalent to $a^{k_A} r < (ar)^{k_A}$, which is clear.

The function $g \circ g$ has fixed points at $\log k_A$, $x_0$, and $k_A \log k_A$. By \cref{lem:exppts} it cannot have any more fixed points, so $g(g(x)) - x \neq 0$ for all $x_0 < x < k_A \log k_A$. Thus, to show that $g(g(\delta_B)) - \delta_B > 0$ it suffices to show that there exists some $x_0 < \delta < k_A\log k_A$ with $g(g(\delta)) - \delta > 0$.

To do this, we will show that the derivative of $g\circ g$ at $k_A \log k_A$ is strictly less than $1$. To start, observe that $g'(x) = -\alpha g(x)$. Therefore, the derivative of $g(g(x))$ at $k_A\log k_A$ is, by the chain rule,
\[g'(\log k_A) g'(k_A\log k_A) = \alpha^2 (\log k_A)(k_A\log k_A) = \frac{k_A\log (k_A)^2}{(k_A-1)^2}.\]
We wish to show that $k_A \log(k_A)^2 < (k_A-1)^2$. Since equality holds at $k_A = 1$, we only need to show that the derivative of the left hand side is less than the derivative of the right hand side for all $k_A > 1$, i.e., $\log(k_A)^2 + 2 \log(k_A) < 2(k_A-1)$. Equality again holds at $k_A = 1$, so by taking the derivative again we need to show
\[\frac{2(\log k_A + 1)}{k_A} < 2 \iff k_A - 1> \log k_A.\]
This is well-known.
\end{proof}

We have now shown that $\xi_m(k_A) < (\log k_A)^{k_A-1}$ for a wide variety of $k_A$: specifically, if any counterexamples greater than $2$ exist, they must be primes at least $5$. Especially given that \cref{thm:limxim} implies that only finitely many counterexamples do exist, it is attractive to conjecture that in fact none exist. However, constructions for small primes at least $5$ remain elusive.

\subsection{\texorpdfstring{$\xi^*_m(k_A)$}{ξ*\_m(k\_A)} and growth rates}
\begin{defn}
Let $\xi^*_m(k_A)$ be the infimum value of $\Delta_B \log(\Delta_A)^{k_A}/k_B^{k_A}$ over all $\Delta_A,\Delta_B,k_B$ such that $K_{\Delta_B,\Delta_A}$ is not $(k_A,k_B)$-choosable.
\end{defn}

In fact, since for such $\Delta_A,\Delta_B,k_B$ we have \[\Delta_B \log(\Delta_A)^{k_A}/k_B^{k_A} \geq \xi_m(k_A) \log(\Delta_A),\] we have the following result:
\begin{prop} \label{prop:ximxim}
$\xi_m(k_A)\log(k_A) \leq \xi^*_m(k_A)$.
\end{prop}

We will now prove bounds on the growth rates of $\xi_m(k_A)$ and $\xi^*_m(k_A)$. First, by \cref{prop:ximbounds}, $\xi_m(k_A) \geq \alpha(k_A)$. Here, we determine the asymptotic behavior of $\alpha(k_A)$.

\begin{prop} \label{prop:alphabound}
$\lim_{k_A \to \infty} \log(\alpha(k_A))/k_A = 0$.
\end{prop}
\begin{proof}
To prove an upper bound, first note that since $\alpha(k_A) \leq 1$, we have
\[\limsup_{k \to \infty} \frac{\log(\alpha(k_A))}{k_A} \leq 0.\]

Now, set $u = 1/k_A$. Since $\log(1-x) \geq -2x$ for sufficiently small positive $x$, for small $u$ we have
\[\log(1 - u + u \log u) \geq -\frac{2 \log(ek_A)}{k_A}.\]
Therefore, for sufficiently large $k_A$,
\[\alpha(k_A) \geq u(1 - u + u \log u)^{k_A-1} \geq \frac{(1 - u + u \log u)^{k_A}}{k_A} \geq \frac{1}{e^2k_A^3},\]
which shows $\liminf_{k_A \to \infty} \log(\alpha(k_A))/k_A \geq 0$.
\end{proof}
\begin{prop}\label{prop:loglogbound}
$\limsup_{k_A \to \infty} \log(\xi^*_m(k_A))/k_A \leq \log 2 + \log \log 2$.
\end{prop}
\begin{proof}
Erd\H{o}s-Rubin-Taylor \cite{Erdos1980} showed that $K_{\Delta,\Delta}$ is not $(k_A,k_A)$-choosable if $\Delta = k_A^2 2^{k_A+1}$. Therefore,
\[\xi^*_m(k_A) \leq \frac{k_A^2 2^{k_A+1} ((k_A+1)\log 2 + 2\log k_A)^{k_A}}{k_A^{k_A}} = k_A^22^{k_A+1} \paren*{\log 2 + \frac{\log 2 + 2\log k_A}{k_A}}^{k_A}.\]
The logarithm of this is asymptotic to $(\log 2 + \log \log 2) k_A$, as desired.
\end{proof}

\subsection{Asymptotic existence}
We start with a quick lemma using graph amplification:
\begin{lem} \label{lem:23amplify}
Suppose $K_{\Delta_B,\Delta_A}$ is not $(k_A,k_B)$-choosable. Let $a$ and $b$ be nonnegative integers and let $r = 2^a3^b$. Then $K_{(6\Delta_B)^r/6, (6\Delta_A)^r/6}$ is not $(rk_A,rk_B)$-choosable.
\end{lem}
\begin{proof}
Observe that it suffices to show the result for $r = 2$ and $r=3$, after which the result for general $r$ will follow by chaining steps that multiply $r$ by $2$ or $3$.

If $r = 2$, we want to show that $K_{6\Delta_B^2, 6\Delta_A^2}$ is not $(2k_A,2k_B)$-choosable. To do this, note that $(K_{\Delta_B,\Delta_A})^{\curlyvee 2} = K_{2\Delta_B, \Delta_A^2}$ is not $(k_A,2k_B)$-choosable. Applying another $2$-fold blowup with the roles of $A$ and $B$ reversed implies that $K_{4\Delta_B^2, 2\Delta_A^2}$ is not $(2k_A,2k_B)$-choosable, as desired.

Similarly, we can show that $K_{27\Delta_B^3, 3\Delta_A^3}$ is not $(3k_A, 3k_B)$-choosable. Since $27 < 6^2$, we are done.
\end{proof}

\begin{proof}[Proof of \cref{thm:limxim}]
First observe that by applying \cref{prop:ximxim} and \cref{prop:alphabound}, the quantity $\liminf_{k_A\to\infty} \log (\xi^*_m(k_A))/k_A$ exists. We will find a function $f(x, k_A)$ with $\lim_{k_A\to\infty} f(x, k_A) = x$ such that if \[\frac{\log (\xi^*_m(k_0))}{k_0} < x,\] then we have \[\limsup_{k_A\to\infty} \frac{\log (\xi^*_m(k_A))}{k_A} \leq f(x, k_0).\] Thus, for all $x > \liminf_{k_A\to\infty} \log (\xi^*_m(k_A))/k_A$, we have $\limsup_{k_A\to\infty} \log (\xi^*_m(k_A))/k_A \leq x$, which will finish the proof.

For a real number $s$, let $r(s)$ be the smallest number of the form $2^a3^b$ that is at least $s$, where $a$ and $b$ are nonnegative integers. Since $\log(3)/\log(2)$ is irrational, $\lim_{s\to\infty} r(s)/s = 1$.

Now take some $x$ and $k_0$ such that $\log (\xi^*_m(k_0))/k_0 < x$. Thus there exist $\Delta_A,\Delta_B,k_B$ such that $K_{\Delta_B, \Delta_A}$ is not $(k_0, k_B)$-choosable and
\[\frac{\Delta_B \log(\Delta_A)^{k_0}}{k_B^{k_0}} < e^{xk_0} \iff \frac{\Delta_B^{1/k_0} \log(\Delta_A)}{k_B} < e^x.\]
For a positive integer $k_A$, note that by \cref{lem:23amplify}, the graph $K_{(6\Delta_B)^{r(k_A/k_0)}, (6\Delta_A)^{r(k_A/k_0)}}$ is not $(k_0 r(k_A/k_0), k_B r(k_A/k_0))$-choosable and therefore not $(k_A, k_B r(k_A/k_0))$-choosable. Therefore
\[\frac{\log(\xi^*_m(k_A))}{k_A} \leq \log \frac{(6\Delta_B)^{r(k_A/k_0)/k_A} r(k_A/k_0) \log(6\Delta_A)}{k_B r(k_A/k_0)},\]
where the right-hand side approaches
\[\log \frac{(6\Delta_B)^{1/k_0} \log(6\Delta_A)}{k_B} < x + \log\paren*{ 6^{1/k_0} \frac{\log(6k_0)}{\log(k_0)}}\]
as $k_A \to \infty$ (we used the fact that $\Delta_A \geq k_0$). Therefore
\[f(x, k_A) = x + \log\paren*{ 6^{1/k_A} \frac{\log(6k_A)}{\log(k_A)}}\]
is the desired function, completing the proof.
\end{proof}

By applying \cref{prop:ximxim}, \cref{prop:alphabound}, and \cref{prop:loglogbound}, we know that the limit $\lim_{k_A\to\infty} \log(\xi^*_m(k_A))/k_A$ is at least $0$ but at most $\log 2 + \log \log 2 \approx 0.3266$. Despite appearances, the question of whether the limit $\lim_{k_A\to\infty} \log(\xi^*_m(k_A))/k_A$ is zero or positive is relatively unimportant compared to determining the actual value of the limit. To see this, observe that changing the base of the logarithm used to define $\xi^*_m(k_A)$ will multiply values by an exponential in $k_A$, changing the value of the limit by a constant.

It is natural to make the following conjecture:
\begin{conj}
$\lim_{k_A\to\infty} \log(\xi_m(k_A))/k_A$ exists and is equal to $\lim_{k_A\to\infty} \log(\xi^*_m(k_A))/k_A$.
\end{conj}
However, applying similar methods to attack this problem does not directly work, since one cannot rule out the case where the values of $\Delta_A$ witnessing a low value of $\xi_m(k)$ grow extremely quickly. Therefore, resolving this conjecture appears to require a deeper understanding of list coloring beyond graph amplification techniques.

\section{On General Bipartite Graphs} \label{sec:ext}
When $G$ is allowed to be any bipartite graph with the maximum degrees $\Delta_A$ and $\Delta_B$, the combinatorial tools required to prove the $(k_A, k_B)$-choosability of many relevant bipartite graphs are far more elusive.
\begin{defn}
Define $\xi_g(k_A)$ as the infimum value of $\xi$ over all $k_B$ and bipartite graphs $G$ with maximum degrees $\Delta_A, \Delta_B$ that are not $(k_A, k_B)$-choosable.
\end{defn}
The main difference between $\xi_g(k_A)$ and $\xi_m(k_A)$ is that, in the general bipartite case, there is no analogue to \cref{prop:2smoothupper}, meaning that it is possible that $\xi_g(k_A) = 0$. It is unknown whether this is the case for any $k_A$, but proving the contrary for any $k_A > 1$ seems to be very difficult:
\begin{prop}
Suppose $\xi_g(k_A) > 0$. Then, for any bipartite graph $G$, we have $\ch(G) = O(\Delta^{1/k_A} \log(\Delta)^{1-1/k_A})$.
\end{prop}
\begin{proof}
By the definition of $\xi_g$, if
\[\frac{\Delta \log(\Delta)^{k_A-1}}{k_B^{k_A}} < \xi_g(k_A),\]
then $G$ is $(k_A, k_B)$-choosable, and thus $(k', k')$-choosable for $k' = \max(k_A, k_B)$. So
\[\ch(G) \leq \max\paren*{k_A, \xi_g(k_A)^{-1/k_A} \Delta^{1/k_A} \log(\Delta)^{1-1/k_A} +1},\]
as desired.
\end{proof}
Therefore, a proof that $\xi_g(k_A) > 0$ for any $k_A > 1$ would imply major progress towards resolving \cref{conj:ak}.

Interestingly, even if the positivity of $\xi_g(k_A)$ is assumed, proving analogues to \cref{thm:lopsidedasymptotic} and \cref{thm:eqasymptotic} would not result from a simple reapplication of the techniques of this paper. The main difficulty in doing so is that our use of graph amplification sometimes does not respect the local structure of the graph; in particular, $\Delta_A(G^{\curlyvee r}) = \Delta_A(G) \abs{B(G)}^{r-1}$.


Finally, more complex properties of the $\xi_g(k_B)$ correspond to generalizations of \cref{conj:ak}. In particular, \cref{conj:xig} posits that $\xi_g(k_A) > 0$ for all $k_A$ and that $\lim_{k_A\to\infty} \log(\xi_g(k_A))/k_A$ exists. We conclude the paper with the following claim:

\begin{thm}
Assume \cref{conj:xig}. Then all three parts of \cref{conj:ack} are true.
\end{thm}
\begin{proof}
We will actually only use the fact that $\xi_g(k_A) > 0$ for all $k_A$ and
\[\liminf_{k_A\to\infty} \log(\xi_g(k_A))/k_A > -\infty,\] 
which is equivalent to $\xi_g(k_A) > c^{k_A}$
for some positive $c$. Therefore, $G$ is $(k_A, k_B)$-choosable whenever
\[\Delta_B^{1/k_A} \log(\Delta_A)^{1-1/k_A} < ck_B.\]
To prove part (a), fix some positive $\eps$. Pick $\Delta_0$ sufficiently large such that $\log \Delta > 1$, $\Delta^{-\eps} < \eps/2$, and $\Delta^{\eps/2}/\log(\Delta) > 1/(c\eps)$ for all $\Delta \geq \Delta_0$. Now assume $\Delta_B \geq \Delta_A > \Delta_0$. Then, if $k_A \geq \Delta_A^\eps$ and $k_B \geq \Delta_B^\eps$, we have
\[\Delta_B^{1/k_A} \log(\Delta_A)^{1-1/k_A} < \Delta_B^{\eps/2} \log(\Delta_B) < c\Delta_B^\eps \leq ck_B,\]
as desired.

To prove part (b), note that since $\log(\Delta_A) \geq \log(2)$, there is some $c'$, namely $c' = c\log (2)$, such that $G$ is $(k_A, k_B)$-choosable whenever $\Delta_B^{1/k_A} \log(\Delta_A) < c' k_B$. This condition is equivalent to $\Delta_B^{1/k_A} \log(\Delta_A^{1/k_B}) < c'$. Therefore, letting $q > 1$ be such that $q \log q < c'$, it is sufficient to show that $\Delta_B^{1/k_A}  , \Delta_A^{1/k_B}\leq q$, which is equivalent to $ k_A \geq \log(\Delta_B)/\log(q)$ and $ k_B \geq \log(\Delta_A)/\log(q)$. Thus $1/\log(q)$ is the desired constant $C$.

Finally, we prove part (c). If $\Delta_A = \Delta_B = \Delta$ and $k_B \geq C(\Delta/\log \Delta)^{1/k_A} \log \Delta$, we have
\[\Delta_B^{1/k_A} \log(\Delta_A)^{1-1/k_A} \leq k_B/C.\]
If $C > 1/c$, then $G$ is $(k_A,k_B)$-choosable.
\end{proof}

\section*{Acknowledgements}
This work was conducted at the REU at the University of Minnesota Duluth, supported by the National Science Foundation grant DMS-1949884 and National Security Agency grant H98230-20-1-0009. The author would like to thank Joseph Gallian for organizing the REU and providing frequent feedback and support. The author would also like to thank Amanda Burcroff for advising and helpful discussions, as well as Mehtaab Sawhney, Sumun Iyer, and the anonymous referees for making valuable suggestions to the manuscript.

\printbibliography

\appendix
\section{Supplemental Proofs}\label{append:ineq}
\subsection{Proof of \texorpdfstring{\cref{lem:tedious}}{Lemma \ref{lem:tedious}}} 
Observe that the inequality holds when $\beta = 0$, so we will be done if we can show that for every $\beta > 0$ the logarithmic derivative of the left hand side is at most the logarithmic derivative of the right hand side. Thus, we want to show that
\[\frac{-\frac{(\gamma - a)^2}{\gamma - b}}{1+\beta\frac{\gamma - a}{\gamma - b}} \leq \frac{-\gamma}{1+\beta} + \frac{a^2/b}{1+\beta a^2/b}.\]
Clearing denominators (all of which are positive), this rearranges to
\[a^2(1+\beta)(\gamma - b + \beta(\gamma - a)) - \gamma(b+\beta a^2)(\gamma - b + \beta(\gamma - a)) + (\gamma-a)^2(1+\beta)(b+\beta a^2) \geq 0.\]
After expanding, we are left with an inequality that is only degree $1$ in $\gamma$, as follows:
\[\gamma((a-b)^2 + \beta a(1-a)(2a-b) + \beta^2 a^2(1-a)) - a^3\beta (1-a)(1+\beta) \geq 0. \tag{$\dagger$} \label{eq:mainmess}\]
The coefficient of $\gamma$ in (\ref{eq:mainmess}) is nonnegative, since it can be rewritten as
\[(1-a)((a-b)^2+\beta a(a-b)+\beta^2 a^2) + a(a-b)^2 + \beta a^2(1-a);\]
thus we only need to check $\gamma = \max(a, b)$ to ensure that the inequality holds for all $\gamma \geq \max(a,b)$. If $a \geq b$ then setting $\gamma = a$ in (\ref{eq:mainmess}) yields
\[a(a-b)(a-b+\beta a(1-a)) \geq 0,\]
which clearly true. If $b \geq a$ then setting $\gamma = b$ in (\ref{eq:mainmess}) yields
\[(b-a)(b(b-a)-\beta a(1-a)(b-a)+\beta^2a^2(1-a)) \geq 0,\]
where the second term is nonnegative as it is equal to
\[ab(b-a)+a(1-a)(b-a)+(1-a)((b-a)^2-\beta a (b-a) + \beta^2 a^2).\]
This concludes the proof.

\subsection{Proof of \texorpdfstring{\cref{lem:exppts}}{Lemma \ref{lem:exppts}}}
We use the \emph{Schwarzian derivative}
\[(Sf)(x) = \frac{f'''(x)}{f'(x)} - \frac{3}{2} \paren*{\frac{f''(x)}{f'(x)}}^2.\]
It is known (see e.g. \cite{schwarz}) that $S(g\circ f)(x) = (Sg)(x) f'(x)^2 + (Sf)(x)$. It is simple to compute $(Sg)(x) = -a^2/2$, so we have $(Sh)(x) < 0$ for all $x$, where $h = g \circ g$. Observe that since $g'(x) < 0$ for all $x$, $h'(x) > 0$ for all $x$.

Suppose there exist at least four distinct solutions to the equation $h(x) = x$. Then, by the mean value theorem there exist $b_1 < b_2 < b_3$ with $h'(b_1) = h'(b_2) = h'(b_3) = 1$. Let $b^*$ be in the interval $[b_1, b_3]$ such that $h'(b^*)$ is minimal. We may further assume that $b^*$ is in the open interval $(b_1, b_3)$, since if $h'$ achieves a minimum at either $b_1$ or $b_3$, we may take $b^* = b_2$.

Now, by the second derivative test, we have $h''(b^*) = 0$ and $h'''(b^*) \geq 0$. This implies that $(Sh)(b^*) \geq 0$, a contradiction.

\end{document}